\documentclass[11pt,english]{article}

\usepackage[T1]{fontenc}
\usepackage[latin9]{inputenc}
\usepackage{geometry}
\usepackage{graphicx}
\usepackage{verbatim}
\usepackage{amsmath}
\usepackage{amssymb}

\makeatletter
\usepackage{times}
\usepackage{algorithm}
\usepackage{algorithmic}
\usepackage{amsthm}
\usepackage[all]{xy}

\makeatletter
 \theoremstyle{plain}
\newtheorem{thm}{Theorem}[section]
  \theoremstyle{plain}
  
 \theoremstyle{definition}
  
  \theoremstyle{definition}
  \newtheorem{defn}[thm]{Definition}
  \theoremstyle{definition}
  
  \theoremstyle{definition}
  
  \theoremstyle{plain}
  
  \theoremstyle{plain}
  \newtheorem{lem}[thm]{Lemma}
    \theoremstyle{plain}
  
 \makeatother

\renewcommand{\epsilon}{\varepsilon}
\renewcommand{\u}{\textbf{u}}

\newcommand{\R}{\mathbb{R}}

\usepackage{babel}
\makeatother

\begin{document}

\title{A Strong Maximum Principle for Parabolic Systems in a Convex Set with
Arbitrary Boundary}

\author{Lawrence Christopher Evans \footnote{Massachusetts Institute of Technology, Department of Mathematics. Email: lcevans@math.mit.edu}}



\date{$\,$}

\maketitle

\begin{abstract} In this paper we prove a strong maximum principle
for certain parabolic systems of equations.  In particular, our methods
place no restriction on the regularity of the boundary of the convex set in
which the system takes its values, and therefore our results hold for any
convex set. We achieve this through the use of viscosity solutions and
their corresponding strong maximum principle.  \end{abstract}

\thispagestyle{empty}

\setcounter{page}{1}

\section*{Introduction}

In this paper we prove a strong maximum principle for
certain parabolic systems.  To be precise, 
let $X\subseteq\R^n$ be open and connected, and set $\Omega =
X\times(0,\infty)$. Suppose $\u\in C^{2,1}(\Omega;\R^k)\cap
C(\bar{\Omega};\R^k)$ satisfies a parabolic system of equations of the form
\begin{eqnarray}
\label{eqn:eq}
\frac{\partial\u}{\partial t}=D(x,t,\u)\sum_{i,j}a_{ij}(x,t)\frac{\partial^2 \u}{\partial x_i\partial
x_j}+\sum_{i}M_{i}(x,t,\u)\frac{\partial \u}{\partial x_i}+\phi(x,t,\u).
\end{eqnarray}
where the $a_{ij}$ are real-valued, $\phi$ takes values in $\R^k$, and $D(x,t,z)$ and each of the $M_i(x,t,z)$ take values in the space of $k\times k$ matrices. 

We make the following regularity assumptions: As a function of $z$, we assume that $\phi(x,t,z)$, $D(x,t,z)$ and each of the $M_i(x,t,z)$ are Lipschitz continuous uniformly for $(x,t)$ in compact subsets of $\Omega $. We assume also that each of the $a_{ij}(x,t)$ is locally bounded in $\Omega$ and that $D(x,t,z)$ and each of the $M_i(x,t,z)$ are locally bounded in $\Omega\times \R^k$. Finally, we assume that the matrix $\{a_{ij}(x,t)\}$ is symmetric and locally uniformly positive definite in $\Omega$ and the matrix $D(x,t,z)$ is locally uniformly positive definite in $\Omega\times \R^k$.

Next, suppose that $K$ is a closed, convex subset of $\R^k$ which is \emph{compatible} with $\phi$, $D$, and the $M_i$ in the following sense: For all $(x,t)\in\Omega$, $v\in\partial K$ and each vector $\nu$ which is inward pointing at $v$ (See \S\,1 for the definition of ``inward pointing at $v$.''), we have that $\phi(x,t,v)\cdot\nu\geq 0$ and that $\nu$ is a left eigenvector of $D(x,t,v)$ and each of the $M_i(x,t,v)$.

We will show that if 
\begin{eqnarray}
\label{eqn:WMP}
\u(x,t)\in K\text{ for all }(x,t)\in\Omega ,
\end{eqnarray}
then
\begin{eqnarray}
\label{eqn:SMP}
\u(x_0,t_0)\in\partial K\text{ for some }
(x_0,t_0)\in\Omega \implies \u(x,t)\in\partial K\text{ for all }(x,t)\in
X\times [0,t_0],
\end{eqnarray}
and we will refer to the implication in (\ref{eqn:SMP}) as {\bf the strong
maximum principle} for solutions to (\ref{eqn:eq}).
Under circumstances when solutions to (\ref{eqn:eq}) satisfy a {\bf weak maximum
principle} of the form
\begin{eqnarray}
\label{WMP}
\u(x,t)\in K\text{ for }(x,t)\in\partial \Omega \implies \u(x,t)\in K\text{
for }(x,t)\in\Omega ,
\end{eqnarray}
it is obvious that (\ref{eqn:WMP}) can be replaced by
$$
\u(x,t)\in K\text{ for } (x,t)\in\partial \Omega.
$$

In his 1975 paper ~\cite{key-6}, H. Weinberger considers the case where $D\equiv I$ and the $M_i$ are real-valued functions. In this case, Weinberger states and proves the strong maximum principle for parabolic systems under a regularity condition on the convex set $K$ which he calls the ``slab condition''. In their 1977 paper ~\cite{key-1}, K. Chueh, C. Conley, and J. Smoller prove, under mild regularity assumptions, a weak maximum principle of the form (\ref{WMP}) for the system of partial differential equations (\ref{eqn:eq}), under the same compatibility condition we have assumed on $K$, $\phi$, $D$, and the $M_i$. In his 1990 paper ~\cite{key-5}, X. Wang, gives a geometric proof of the strong maximum principle for (\ref{eqn:eq}), following Weinberger's arguments, in the case that the boundary of $K$ is $C^2$. In this case, the distance function, $d$, to the boundary of $K$ is $C^2$ in $K$ (at least near the boundary), in which case he proves the strong maximum principle by applying the classical strong maximum principle to $d(\u(x,t))$.

The argument presented in this paper is similar to X. Wang's; we too will apply the classical strong maximum principle to $d(\u(x,t))$. However, we are able to remove all regularity requirements on $\partial K$ by showing that even if $d(\u(x,t))$ is not twice differentiable, it is still a viscosity super solution of the appropriate partial differential equation (See ~\cite{key-2} for an introduction to the theory of viscosity solutions). We can then apply a strong maximum principle for viscosity solutions to parabolic partial differential equations to complete the proof of (\ref{eqn:SMP}). The strong maximum principle we need is provided by F. Da Lio in ~\cite{key-3}, wherein Da Lio extends the classic argument given by L. Nirenberg in ~\cite{key-4} to the case of viscosity solutions.

\section{Proof of (\ref{eqn:SMP})}  

We begin with the following definitions:

\begin{defn}
Given a convex set $K$ and a boundary point $v\in\partial K$, a function $\ell:K\rightarrow \R$ is called a \emph{supporting affine functional of }$K$\emph{ at }$v$ if $\ell(z)\geq 0$ for all $z\in K$, $\ell(v)=0$, and $|\nabla\ell(z)|\equiv|\nabla\ell|=1$ (That is, $\ell(z)$ is the distance to a supporting hyperplane of $K$ at $v$, which, by abusing notation, we also denote by $\ell$). 

We say that a function $\ell:K\rightarrow \R$ is a \emph{supporting affine functional} of $K$ if it is a supporting affine functional of $K$ at some $v\in\partial K$.
\end{defn}
\begin{defn}
For $v\in\partial K$, a vector $\nu$ is an \emph{inward pointing vector} at $v\in\partial K$ if there exists a supporting affine functional $\ell(z)$ at $v$ such that $\nabla\ell=\nu$ (That is, if there exists a supporting hyperplane $\ell$ of $K$ at $v$ whose ``inward pointing'' unit normal vector is $\nu$).
\end{defn}

Note that in our definitions we include the assumption that an inward pointing vector has unit length and that a supporting affine functional has unit-length gradient.  For geometric reasons, these are nice assumptions to make as they allow for the intuitively appealing interpretation of $\ell(z)$ as the distance to a supporting hyperplane $\ell$.

Given $v\in \partial K$, we use $L_v$ to denote the set of all the
supporting affine functionals $\ell$ of $K$ at $v$.  By well known results (e.g.,
the Hahn-Banach Theorem), $L_v\neq\emptyset $ for every $v\in\partial K$. For $z\in K$ we let $d(z)=\inf\{|z-v|:v\in\partial K\}$ denote the distance from $z$ to the boundary of $K$.

First we present a lemma from convex analysis.
\begin{lem}
\label{convexlemma}
For any convex set $K$,
$$
d(z)=\inf_\ell \ell(z),
$$
where the infimum is taken over all supporting affine functionals $\ell$. Furthermore, for each $z\in K$, there is a (possibly non-unique) point $v\in\partial K$ and an $\ell\in L_v$ such that
$$
d(z)=|z-v|=\ell(z).
$$
In fact, $L_v$ consists of just this one supporting affine functional $\ell$.
\end{lem}
\begin{proof}
It is easy to see geometrically that $d(z)\leq \ell(z)$ for each $\ell$: Given $\ell$, $\ell(z)=|z-\bar{z}|$, where $\bar{z}$ is the projection of $z$ onto the supporting hyperplane determined by $\ell$. The line from $z$ to $\bar{z}$ intersects $\partial K$ at some point $w$. Thus
$$
d(z)\leq|z-w|\leq|z-\bar{z}|=\ell(z).
$$

Next we show that the infimum over $\ell$ is in fact achieved. Using $B(z,r)$ to denote the closed ball with radius $r$ centered at $z$, we have that $d(z)=\inf\{|z-w|:w\in\partial K\}=\inf\{|z-w|:w\in\partial K\cap B(z,2d(z))\}$. By compactness, this infimum is achieved and we have $d(z)=|z-v|$ for some $v\in\partial K$. 

Take $\ell\in L_v$. We have shown that $d(z)\leq \ell(z)$. Since $v\in \ell$, we also have that $\ell(z)=\text{dist}(z,\ell)\leq |z-v|=d(z)$. And so it follows that $d(z)=\ell(z)$.

So for all $\ell\in L_v$ we have that $\ell(z)=|z-v|$ which means that the line from $z$ to $v$ is normal to the hyperplane $\ell$. As there is only one hyperplane with this normal which touches $v$, we have that $L_v$ is just the singleton $\{\ell\}$.

\end{proof}

In the remainder of the paper we will use the notation $\bar{d}(x,t)$ for $d(\u(x,t))$ and $\bar{\ell}(x,t)$ for $\ell(\u(x,t))$. The following is the key result proved in this section.

\begin{thm}
\label{thm:visc} 
$\bar{d}(x,t)$ satisfies, in the viscosity sense, a parabolic equation of the form
$$
\frac{\partial \bar{d}}{\partial t}-\sum_{i,j}\alpha_{ij}(x,t)\frac{\partial^2\bar{d}}{\partial x_i \partial x_j}-\sum_i \beta_i(x,t)\frac{\partial \bar{d}}{\partial x_i}+\gamma(x,t)\bar{d}\geq 0,
$$
where $\gamma(x,t)\geq 0$ in $\Omega$, each of the $\alpha_{ij},\beta_i,$ and $\gamma$ are locally bounded in $\Omega$, and $\{\alpha_{ij}\}$ is locally uniformly positive definite in $\Omega$.
\end{thm}

\begin{proof}
Call a quadruple $(x,t,v,\ell)$ \emph{nice} if 
\begin{enumerate}
\item $(x,t)\in\Omega$, $v\in\partial K$, and $\ell$ is a supporting affine functional of $K$ at $v$.
\item $d(\u(x,t))=|\u(x,t)-v|$.
\end{enumerate} 

By the previous lemma, we know that for each $(x,t)$ we can find at least one $v\in\partial K$ satisfying the second condition and this in turn will determine a unique choice of $\ell$ satisfying the first condition. We then also have that $\ell(\u(x,t))=|\u(x,t)-v|$. 

For any nice quadruple we have that

\begin{eqnarray*}
\bar{\ell}_t&=&\nabla \ell\cdot\u_t\\
                 &=&\nabla \ell\cdot[D(x,t,\u)\sum_{i,j}a_{ij}(x,t)\u_{x_ix_j}+\sum_{i}M_{i}(x,t,\u)\u_{x_i}+\phi(x,t,\u)]\\
                 &=&\nabla \ell\cdot[D(x,t,\u)\sum_{i,j}a_{ij}(x,t)\u_{x_ix_j}+\sum_{i}M_{i}(x,t,\u)\u_{x_i}+\phi(x,t,\u)]\\
                 & &-\nabla \ell\cdot[D(x,t,v)\sum_{i,j}a_{ij}(x,t)\u_{x_ix_j}+\sum_{i}M_{i}(x,t,v)\u_{x_i}+\phi(x,t,v)]\\
                 & &+\nabla \ell\cdot[D(x,t,v)\sum_{i,j}a_{ij}(x,t)\u_{x_ix_j}+\sum_{i}M_{i}(x,t,v)\u_{x_i}+\phi(x,t,v)]\\
                 &\geq&-\left[c(x,t)\|\sum_{i,j}a_{ij}(x,t)\u_{x_ix_j}\|+\sum_i m_i(x,t)\left\|\u_{x_i}\right\|+p(x,t)\right]|\u-v|\\
                 & &+\mu(x,t,v,\nabla \ell)\sum_{i,j}a_{ij}(x,t)\left( \nabla\ell\cdot \u_{x_ix_j}\right)+\sum_i\lambda_i(x,t,v,\nabla\ell)\left( \nabla\ell\cdot \u_{x_i}\right)+0\\
                 &=&-\left[c(x,t)\|\sum_{i,j}a_{ij}(x,t)\u_{x_ix_j}\|+\sum_i m_i(x,t)\left\|\u_{x_i}\right\|+p(x,t)\right]\bar{\ell}\\
                 & &+\mu(x,t,v,\nabla \ell)\sum_{i,j}a_{ij}(x,t)\bar{\ell}_{x_ix_j}+\sum_i\lambda_i(x,t,v,\nabla\ell)\bar{\ell}_{x_i},
\end{eqnarray*}
where $c(x,t)$, $m_i(x,t)$, and $p(x,t)$ are the Lipschitz constants for $D(x,t,\cdot)$, $M_i(x,t,\cdot)$, and $\phi(x,t,\cdot)$, respectively, and $\mu$ and $\lambda_i$ are the eigenvalues for $D$ and $M_i$ at $v$ corresponding to the left eigenvector $\nabla \ell$ (recall that $\nabla \ell$ is inward pointing at $v$). Note that in the inequality step we have used the fact that $\phi(x,t,v)\cdot\nabla\ell\geq 0$ which follows from the fact that $\nabla\ell$ is inward pointing at $v$ and from our compatibility assumption on $\phi$ and $K$.

Next, let $\gamma(x,t)=\left[c(x,t)\|\sum_{i,j}a_{ij}(x,t)\u_{x_ix_j}\|+\sum_i m_i(x,t)\left\|\u_{x_i}\right\|+p(x,t)\right]$. It follows from our regularity assumptions on $D$, $M_i$, and $\phi$ that $c(x,t)$, $m_i(x,t)$, and $p(x,t)$ are locally bounded. Since $\u\in C^{2,1}(\Omega)$, its first and second spatial derivatives are locally bounded, and so, since the $a_{ij}(x,t)$ are locally bounded, it follows that $\gamma(x,t)$ is locally bounded. As each of the Lipschitz constants is positive, it is clear that $\gamma(x,t)\geq 0$.

Thus we have shown that for each nice quadruple $(x,t,v,\ell)$,
\begin{eqnarray}
\label{ellPDE}
\bar{\ell}_t-\mu(x,t,v,\nabla \ell)\sum_{i,j}a_{ij}(x,t)\bar{\ell}_{x_ix_j}-\sum_i\lambda_i(x,t,v,\nabla\ell)\bar{\ell}_{x_i}+\gamma(x,t)\bar{\ell}\geq 0.
\end{eqnarray}

As remarked at the outset of this proof, given $(x,t)\in\Omega$ we can always find a $v\in\partial K$ and $\ell$ to form a nice quadruple $(x,t,v,\ell)$. In general there will be more than one way to extend $(x,t)$ into a nice quadruple, but for our purposes we only care that an extension is possible.

To avoid the axiom of choice, we now describe a method of choosing an extension: Given $(x,t)\in\Omega$ we need that $v\in\partial K$ be such that $|\u(x,t)-v|=d(\u(x,t))$. The set of $v$ which satisfy this relation is a closed and bounded and hence compact subset of $\R^k$. We first look at $v$ in this set with smallest first component. If there is a unique such $v$ we choose it. Otherwise, among those $v$ we look for the one with smallest second component and we continue this algorithm until we find a unique $v$. We denote this $v$ by $v_{(x,t)}$ to make clear its dependence on $(x,t)$. Lastly, once we have $v$, $\ell$ is uniquely determined (see Lemma \ref{convexlemma}) and we denote it by $\ell_{(x,t)}$.

Next we let $\tilde{\mu}(x,t)=\mu(x,t,v_{(x,t)},\nabla\ell_{(x,t)})$ and $\tilde{\lambda_i}(x,t)=\lambda_i(x,t,v_{(x,t)},\nabla\ell_{(x,t)})$. We claim that $\tilde{\mu}(x,t)$ and the $\tilde{\lambda}_i(x,t)$ are locally bounded. This is true as for any compact set $C\subset \Omega$,
\begin{eqnarray*}
\sup_{(x,t)\in C}\tilde{\mu}(x,t)&\leq& \sup_{(x,t)\in C}\sup_{\{v:|\u-v|=d(\u)\}}\mu(x,t,v,\nabla\ell)\\
                                 &\leq& \sup_{(x,t)\in C}\sup_{\{v:|\u-v|=d(\u)\}}\|D(x,t,v)\|<\infty,
\end{eqnarray*}
where $\|D\|$ denotes the operator norm of $D$ which is locally bounded by our regularity assumption on $D$. The same argument works for the $\tilde{\lambda}_i$.

We next claim that $\tilde{\mu}(x,t)$ is uniformly bounded away from $0$ on compact sets. This is true as for any compact set $C\subset \Omega$,
\begin{eqnarray*}
\inf_{(x,t)\in C}\tilde{\mu}(x,t)&\geq& \inf_{(x,t)\in C}\inf_{\{v:|\u-v|=d(\u)\}}\mu(x,t,v,\nabla\ell)\\
                                 &\geq& \inf_{(x,t)\in C}\inf_{\{v:|\u-v|=d(\u)\}}\Lambda_1(D(x,t,v))>0,
\end{eqnarray*}
where we denote by $\Lambda_1(D)$ the smallest eigenvalue of the positive definite matrix $D$. Here we have used our assumption that $D(x,t,v)$ is uniformly positive definite on compact sets.

Finally we reach the crux of our argument. We claim that $\bar{d}$ solves
$$
\bar{d}_t-\tilde{\mu}(x,t)\sum_{i,j}a_{ij}(x,t)\bar{d}_{x_ix_j}-\sum_i\tilde{\lambda}_i(x,t)\bar{d}_{x_i}+\gamma(x,t)\geq 0
$$
in the viscosity sense.

Suppose that $\psi\in C^\infty(\Omega)$ touches $\bar{d}$ from below at the point $(\hat{x},\hat{t})$, i.e. 
$$
\psi(\hat{x},\hat{t})=\bar{d}(\hat{x},\hat{t}),\text{  and }\psi(x,t)\leq \bar{d}(x,t)\text{  in a neighborhood of }(\hat{x},\hat{t}).
$$
We extend the point $(\hat{x},\hat{t})$ to the nice quadruple $(\hat{x},\hat{t},v_{(\hat{x},\hat{t})},\ell_{(\hat{x},\hat{t})})$. Since by Lemma \ref{convexlemma}, $d(z)\leq \ell_{(\hat{x},\hat{t})}(z)$, we then have that
$$
\psi(\hat{x},\hat{t})=\bar{\ell}_{(\hat{x},\hat{t})}(\hat{x},\hat{t}),\text{  and }\psi(x,t)\leq \bar{\ell}_{(\hat{x},\hat{t})}(x,t)\text{  in a neighborhood of }(\hat{x},\hat{t}).
$$
Therefore, $\psi$ also touches the function $\bar{\ell}_{(\hat{x},\hat{t})}$ from below at $(\hat{x},\hat{t})$, and so 
$$
\frac{\partial}{\partial t}\psi=\frac{\partial}{\partial t}\bar{\ell}_{(\hat{x},\hat{t})}(\hat{x},\hat{t}),\text{  and }\Delta_x\psi\leq\Delta_x\bar{\ell}_{(\hat{x},\hat{t})}(\hat{x},\hat{t}).
$$ 
It then follows from (\ref{ellPDE}) (and recalling the definitions of $\tilde{\mu}$ and $\tilde{\lambda_i}$) that
$$
\psi_t-\tilde{\mu}(x,t)\sum_{i,j}a_{ij}(x,t)\psi_{x_ix_j}-\sum_i\tilde{\lambda}_i(x,t)\psi_{x_i}+\gamma(x,t)\psi\geq 0
$$
Since this is true for all points $(\hat{x},\hat{t})$ and all smooth functions $\psi$ touching $\bar{d}$ from below at $(\hat{x},\hat{t})$, we have shown that $\bar{d}(x,t)$ solves
$$
\bar{d}_t-\tilde{\mu}(x,t)\sum_{i,j}a_{ij}(x,t)\bar{d}_{x_ix_j}-\sum_i\tilde{\lambda}_i(x,t)\bar{d}_{x_i}+\gamma(x,t)\bar{d}\geq 0
$$
in the viscosity sense.

The theorem is now proved by letting $\beta_i=\tilde{\lambda}_i$ and $\alpha_{ij}=\tilde{\mu}a_{ij}$. Note that $\{\alpha_{ij}\}$ is uniformly positive definite on compact sets as $\{a_{ij}\}$ is uniformly positive definite on compact sets and $\tilde{\mu}$ is uniformly bounded away from $0$ on compact sets.

\end{proof}

In order to complete the proof of (\ref{eqn:SMP}), we need the following version of the strong maximum principle for supersolutions in the viscosity
sense:
\begin{thm}
\label{thm:vSMP}  
Suppose that $f\in C\bigl(\bar{\Omega} ;[0,\infty
)\bigr)$ satisfies 
$$
\frac{\partial f}{\partial t}-\sum_{i,j}\alpha_{ij}(x,t)\frac{\partial^2 f}{\partial x_i \partial x_j}-\sum_i \beta_i(x,t)\frac{\partial f}{\partial x_i}+\gamma(x,t)f\geq 0
$$
in the viscosity sense.  Suppose further that in $\Omega$, $\gamma \geq 0$, each of the $\alpha_{ij},\beta_i,$ and $\gamma$ are locally bounded, and $\{\alpha_{ij}\}$ is locally uniformly positive definite.  Then
$$
f(x_0,t_0)=0\text{ for some } (x_0,t_0)\in\Omega \implies f(x,t)=0\text{ for all }
(x,t)\in\bar{\Omega} \text{ with } t\leq t_0.
$$
\end{thm}

Note that by applying this theorem to $f=\bar{d}\ge 0$ and using
Theorem \ref{thm:visc}, we get (\ref{eqn:SMP}) as an immediate consequence.

\begin{proof} 
By linearity, we have that $-f$ solves
\begin{eqnarray}\label{eqn:star}
\frac{\partial f}{\partial t}-\sum_{i,j}\alpha_{ij}(x,t)\frac{\partial^2 f}{\partial x_i \partial x_j}-\sum_i \beta_i(x,t)\frac{\partial f}{\partial x_i}+\gamma(x,t)f\leq 0
\end{eqnarray}
in the viscosity sense, and by our non-negativity assumption on $f$, we have that $-f$ achieves a non-negative maximum at $(x_0,t_0)$. If $-f\in C^{2,1}(\Omega;[0,\infty))$ were a classical solution to (\ref{eqn:star}), it would immediately follow from the classical strong maximum principle that $-f\equiv f\equiv 0$ for all $(x,t)\in\Omega$ with $t\leq t_0$.

As $-f$ only solves (\ref{eqn:star}) in the viscosity sense, we instead rely on the extension of the proof of the classical strong maximum principle given by Nirenberg in ~\cite{key-4} to the case of viscosity solutions, provided by F. Da Lio in ~\cite{key-3}. Da Lio's assumptions are more general than we need and apply to viscosity solutions of more general PDE of the form
$$
G(x,t,f,f_t,D_x f,D^2_{xt} f)\leq 0
$$
under the assumption that $G$ is continuous. We don't have continuity for the coefficients of our PDE (\ref{eqn:star}), and so our PDE does not directly fall under the assumptions of Da Lio's result. Nevertheless, because our PDE (\ref{eqn:star}) is of such a simple form, Da Lio's method of extending Nirenberg's proof can still be applied to our PDE, and we achieve the same result that $-f\equiv f\equiv 0$ for all $(x,t)\in\Omega$ with $t\leq t_0$.
\end{proof}

\section{Acknowledgments}  
I thank my advisor, Professor Daniel W. Stroock, for his numerous helpful conversations and suggestions. I also thank Professor David Jerison for his helpful advice.

\end{document}